\documentclass[12pt]{amsart}
\usepackage{enumerate}
\usepackage{geometry}
\usepackage{amsmath,amssymb,amsthm,mathrsfs,amsfonts}
\usepackage{graphicx}
\usepackage{tikz}
\usetikzlibrary{shapes}
\usetikzlibrary{plotmarks}
\usetikzlibrary{arrows}
\usetikzlibrary{positioning}
\usepackage{tkz-graph}
\makeatletter
\@namedef{subjclassname@2010}{%
  \textup{2010} Mathematics Subject Classification}
\makeatother
\newtheorem{thm}{Theorem}[section]

\newtheorem{lem}[thm]{Lemma}
\newtheorem{prop}[thm]{Proposition}
\newtheorem{mthm}{Main Theorem}

\theoremstyle{definition}
\newtheorem{defn}[thm]{Definition}
\newtheorem{rem}[thm]{Remark}

\numberwithin{equation}{section}

\newcommand\myeq{\mathrel{\stackrel{\makebox[0pt]{\mbox{\normalfont\tiny def}}}{=}}}

\def\d{{Dom}}
\def\n{\mathrm{nodes}}
\def\e{\mathrm{edges}}

\def\r{{Rng}}
\def\a#1{\mathfrak{#1}}

\begin{document}
%%%%% To ease editing, for IMPAN journals add:

\baselineskip=17pt

%%%%%%%%%%%

%% In the running head, replace first names by initials 
%% and give an abbreviation of the title.
\title[Geometrical representation theorems]{Geometrical representation theorems for cylindric-type algebras}
\author[M. Khaled]{Mohamed Khaled}
\address{Mohamed Khaled\\ \ Bahcesehir University, Faculty of Engineering and Natural Sciences, Istanbul, Turkey \ \ \& \ \ Alfr\'ed R\'enyi Institute of Mathematics, Hungarian Academy of Sciences, Budapest, Hungary.}
%\email{khaled.mohamed@renyi.mta.hu}
\email{mohamed.khalifa@eng.bau.edu.tr}
\author[T. Sayed Ahmed]{Tarek Sayed Ahmed}
\address{Tarek Sayed Ahmed\\ \ Department of Mathematics, Faculty of Science, Cairo University, Giza, Egypt.}
\email{rutahmed@gmail.com}
\date{}
%\support{Include acknowledgement of support here}
\begin{abstract}
In this paper, we give new proofs of the celebrated Andr\'eka-Resek-Thompson representability results of certain axiomatized cylindric-like algebras. Such representability results provide completeness theorems for variants of first order logic, that can also be viewed as multi-modal logics. The proofs herein are combinatorial and we also use some techniques from game theory.
\end{abstract}
\maketitle
\section{Introduction}
Stone's representation theorem for Boolean algebras can be formulated in two, essentially equivalent ways. Every Boolean algebra is isomorphic to a field
of sets, or the class of Boolean set algebras can be axiomatized by a finite set of equations. As is well known, Boolean algebras constitute the algebraic
counterpart of propositional logic. Stone's representation theorem, on the other hand, is the algebraic equivalent of the completeness theorem for propositional logic.

Throughout, fix a finite ordinal $n\geq 2$. Cylindric algebras of dimension $n$ were introduced by A. Tarski \cite{hmt1} as the algebraic counterpart of first order logic restricted to $n$-many variables. %These algebras are Boolean algebras enriched with finitely many operators (cylindrifiers) and finitely many constant symbols (diagonal elements). 
Unfortunately, not every abstract cylindric algebra is representable as a field of sets, where the extra non-Boolean operations of cylindrifiers and diagonal
elements are faithfully represented by  projections and equalities. This is basically a reflection of the essential incompleteness of the finite variable fragments of first order logics.

In the present paper, we consider some variants of these cylindric algebras that were shown to have representation theorems. Thus, the corresponding logics are variants of first order logic that have completeness theorems. Such logics can be also viewed as multi dimensional modal logics, c.f. \cite{marxphd} and \cite{marx}. The following axioms are essentially due to D. Resek and R. Thompson \cite{resthm}.
\begin{defn}The class $\mathrm{RC}_n$ is defined to be the class of all algebras of the form $\a{A}=\langle A,+,\cdot,-,0,1,c_i,d_{ij}\rangle_{i,j\in n}$ that satisfy the axioms (Ax0) through (Ax7) below. Note that for each $i,j\in n$ with $i\not=j$, we have $s^i_ix\myeq x$ and $s^i_jx\myeq c_i(x\cdot d_{ij})$. 
\begin{enumerate}[(Ax{12})]
\item[(Ax0)] $\langle A,+,\cdot,-,0,1\rangle$ is a Boolean algebra.
\item[(Ax1)] $c_i0=0$, for each $i\in n$.
\item[(Ax2)] $x\leq c_ix$, for each $i\in n$.
\item[(Ax3)] $c_i(x\cdot c_iy)=c_ix\cdot c_iy $, for each $i\in n$. 
\item[(Ax4)] $d_{ii}=1$, for each $i\in n$. 
\item[(Ax5)] $d_{ik}\cdot d_{kj}\leq d_{ij}=d_{ji}=c_kd_{ji}$, for each $i,j,k\in n$ such that $k\not=i,j$.
\item[(Ax6)] $c_i(x\cdot d_{ij})\cdot d_{ij}\leq x$, for each $i,j\in n$ such that $i\not=j$.
\item[(Ax7)] $s^{i_m}_{j_m}c_{k_m}\cdots s^{i_1}_{j_1}c_{k_1}x\cdot\prod\{d_{l\tau(l)}:l\in K\}\leq c_ix$,\\
for each finite $m\geq 1$ and $i_1,\ldots,i_m\in n$ ,$j_1,\ldots,j_m\in n$, $k_1,\ldots, k_m\in n,i\in n$ such that $k_{t+1}\not\in([i_t/j_t]\circ\cdots\circ [i_1/j_1])^*K$, $t<m$, where $\tau=[i_m/j_m]\circ\cdots\circ[i_1/j_1]$ and $K=\{i_1,\ldots, i_m,k_1,\ldots, k_m\}\setminus\{i\}$.
\end{enumerate}
\end{defn}
\begin{rem}Let $\psi$ be any function. We denote its domain and its range by $\d(\psi)$ and $\r(\psi)$, respectively. For any $X\subseteq \d(\psi)$, by $\psi^{*}X$ we mean the set $\{\psi(x):x\in X\}$. Let $i,j\in n$. Then, $[i/j]$ denotes the transformation on $n$ that fixes everything except that it sends $i$ to $j$, while $[i,j]$ is the transformation that fixes everything except that $[i,j](i)=j$ and $[i,j](j)=i$.
\end{rem}
The following axioms are due to R. Thompson and H. Andr\'eka \cite{andthm,and2000}.
\begin{defn} We define $\mathrm{DC}_n=\{\a{A}\in\mathrm{RC}_n:\a{A}\models \text{(Ax8), (Ax9),(Ax{10})}\}$, $\mathrm{SC}_2=\{\a{A}\in \mathrm{DC}_2:\a{A}\models \text{(Ax{11})}\}$ and $\mathrm{SC}_{n}=\{\a{A}\in \mathrm{DC}_{n}:\a{A}\models \text{(Ax{12})}\}$ if $n\geq 3$, where:

\begin{enumerate}[(Ax{12})]
\item[(Ax{8})] $c_ic_j x\geq c_jc_i x\cdot d_{jk}$, for each $i,j,k\in n$ and $k\not=i, j$.
\item[(Ax{9})] $d_{ij}=c_k(d_{ik}\cdot d_{kj})$, for each $i,j,k\in n$ such that $k\not=i,j$.
\item[(Ax{10})] $s^k_is^i_js^j_ms^m_kc_kx=s^k_ms^m_is^i_js^j_kc_kx$, for each $i,j,k,m\in n$, $k\not=i,j,m$ and $m\not=i,j$.
\item[(Ax{11})] $x\cdot-d_{01}\leq c_0c_1(-d_{01}\cdot s^0_1c_1x\cdot s^1_0c_0x)$.
\item[(Ax{12})] $\displaystyle x\leq c_ic_j(s^i_jc_jx\cdot s^j_ic_ix\cdot\prod_{k\in n,k\not=i,j}s^k_is^i_js^j_kc_kx)$, for each $i,j\in n$.
\end{enumerate}
\end{defn}
In \cite{resthm}, D. Resk and R. Thompson proved a representation theorem for the class $\mathrm{RC}_n$. Despite the novelty of their proof, it was quite long and requires a deep knowledge of the literature of algebraic logic. A simpler proof was then provided by H. Andr\'eka (this proof can be found in \cite[Theorem 9.4]{monk}). Andr\'eka's method could also suggest an elegant proof for the representation of the algebras of $\mathrm{DC}_n$ \cite{andthm}. 

The representability of the $\mathrm{DC}_n$-algebras originally is due to R. Thompson, but his original proof was never published. In \cite{and2000}, H. Andr\'eka proved the representability of the algebras in $\mathrm{SC}_n$ by reducing the problem to the case of $\mathrm{DC}_n$ and then applying \cite{andthm}. 
%In this context, a representation theorem states that every abstract algebra can be represented as a concrete algebra. This concrete algebra is a Boolean algebra of sets of sequences of length $n$ and the non-Boolean operations are thus viewed as forming some geometric shapes in the $n$-dimensions. 
Andr\'eka's representing structures were build using the step-by-step construction, which consists of treating defects one by one and then taking a limit where the contradictions disappear. 

In the present paper, we provide new proofs for the representation theorems of the classes defined above. We use games (and networks) as introduced to algebraic logic by R. Hirsch and I. Hodkinson, c.f. \cite{HH} and \cite{HHB}. Our proofs are relatively shorter than all the known proofs. We give direct constructions for all classes, even for $\mathrm{SC}_n$, unlike its original proof in \cite{and2000}.

%The games we use here, which are basically Banach-Mazur games in disguise, are games of infinite lengths between two players $\forall$ and $\exists$. What can be done by step-by-step constructions, can be done by games but not the other way round.  
The translation from step-by-step techniques to games is not a purely mechanical process. This transfer can well involve some ingenuity, in obtaining games that are transparent, intuitive and easy to grasp. The real advantage of the game technique is that games do not only build representations, when we know that such representations exist, but they also tell us when such representations exist, if we do not know a priori that they do.%When we have a step-by-step technique, then we are sure that there is at least one corresponding game. Choosing a simple game is what counts at the end.

Now, we give the formal definition of the representing concrete algebras. We start with the following basic notions. For every $i\in n$ and every two sequences $f,g$ of length $n$, we write $g\equiv_if$ if and only if $g=f(i/u)$, for some $u$, where, $f(i/u)$ is the sequence which is like $f$ except that it's value at $i$ equals $u$. Let $V$ be an arbitrary set of sequences of length $n$. For each $i,j\in n$ and each $X\subseteq V$, we define
$$C_i^{[V]}X=\{f\in V:(\exists g\in X)f\equiv_ig\}.$$ 
$$D_{ij}^{[V]}=\{f\in V:f(i)=f(j)\}.$$
When no confusion is likely, we omit the superscript $[V]$.
\begin{defn}The class of all relativized cylindric set algebras of dimension $n$, denoted by $\mathrm{RCs}_{n}$, is defined to be the class that consists of all subalgebras of the (full) algebras of the form, $$\a{P}(V)\myeq\langle \mathcal{P}(V), \cup,\cap,\setminus,\emptyset, V, C_i^{[V]}, D_{ij}^{[V]}\rangle_{i,j\in n},$$
where $V$ is a non-empty set of sequences of length $n$ and $\mathcal{P}(V)$ is the family of all subsets of $V$. For every $\a{A}\subseteq\a{P}(V)$, the set $V$ is called the unit of $\mathfrak{A}$, while the smallest set $U$ that satisfies $V\subseteq {^nU}$ is called the base of $\mathfrak{A}$, where ${^nU}$ is the set of all sequences of length $n$ whose range is subset of $U$.
\end{defn}
Let $f,g$ be two functions such that $\r(f)\subseteq \d(g)$, then $g\circ f$ denotes the function whose domain is $\d(f)$ and $g\circ f(x)=g(f(x))$. 
\begin{defn}
The class $\mathrm{DCs}_n$ of diagonalizable cylindric set algebras, of dimension $n$, consists of all algebras $\a{A}\in \mathrm{RCs}_n$ whose units $V$ are diagonalizable sets, i.e. for every $f\in V$ and every $i,j\in n$ we have $f\circ[i/j]\in V$. 
\end{defn}
\begin{defn}
The class $\mathrm{SCs}_n$ of locally squares cylindric set algebras, of dimension $n$, consists of all algebras $\a{A}\in \mathrm{DCs}_n$ whose units $V$ are permutable sets, i.e. for every $f\in V$ and every $i,j\in n$ we have $f\circ[i,j]\in V$. 
\end{defn}

In contrast with the literature, other notations for the classes $\mathrm{RCs}_n$, $\mathrm{DCs}_n$ and $\mathrm{SCs}_n$ are $\mathrm{Crs}_n$, $\mathrm{D}_n$ and $\mathrm{G}_n$, respectively. Let $\Lambda_n\myeq {^nn}$ be the  set of all transformations on $n$. Let $\Omega_n\myeq \{\tau\in\Lambda_n:\vert\r(\tau)\vert<n\}$, the set of all transformations on $n$ that are not permutations. Let $V$ be a set of sequences of length $n$. The following characterization is well known and can be verified easily using some simple facts of transformations.
\begin{itemize}
\item $V$ is diagonalizable if and only if $f\circ\tau\in V$, for every $f\in V$ and every $\tau\in\Omega_n$.
\item $V$ is diagonalizable and permutable if and only if $f\circ\tau\in V$, for every $f\in V$ and every $\tau\in\Lambda_n$.
\end{itemize}
For any class $\mathrm{K}$ of algebras, $\mathbf{I}\mathrm{K}$ is the class that consists of all isomorphic copies of the members of $\mathrm{K}$. As we mentioned before, we aim to reprove the following theorem.
\begin{mthm}\label{main}
Let $n\geq 2$ be a finite ordinal and let $\mathrm{K}\in\{\mathrm{RC},\mathrm{DC},\mathrm{SC}\}$. Then, $\mathrm{K}_n=\mathbf{I}\mathrm{Ks}_n$.
\end{mthm}
\section{Preliminary lemmas}
Recall the basic concepts of Boolean algebras with operators (BAO) from the literature, see e.g. \cite{tarski}. For any $\mathfrak{B}\in\mathrm{RC}_n$, let $At(\mathfrak{B})$ be the set of all atoms in $\mathfrak{B}$.

%\begin{defn}\label{perfect}
%An algebra $\a{A}\in\mathrm{RC}_n$ is said to be a perfect algebra if and only if $\a{A}$ is complete, atomic and its conversion and composition are completely additive. 
%\end{defn}
Note that $\mathrm{RC}_n$ (similarly $\mathrm{DC}_n$ and $\mathrm{SC}_n$) is defined by positive equations, the negation does not appear in a non-Boolean axiom. Also note that the cylindrifications are defined to be normal and additive operators. Thus, by \cite[Theorem 2.18]{tarski}, every algebra in $\mathrm{RC}_n$ can be embedded into a complete and atomic algebra in $\mathrm{RC}_n$. Therefore, it is enough to prove that every complete and atomic algebra is representable. We need to prove some auxiliary lemmas that may be interesting in their own. 
\begin{lem}\label{aib} Let $\a{A}\in\mathrm{RC}_n$, let $i\in n$ and let $x,y\in At(\a{A})$. Then, $$x\leq c_iy\iff c_ix=c_iy.$$
\end{lem}
\begin{proof}
We prove the non-trivial direction only. The assumption $x\leq c_iy$ and \cite[Theorem 1.2.9]{hmt1} imply that $c_ix\leq c_iy$. Moreover, 

\begin{eqnarray*}
x\leq c_iy&\implies& x\cdot c_iy\not=0 \\
&\implies& c_i(x\cdot c_iy)\not=0 \hspace{0.5cm}\text{ by \cite[Theorem 1.2.1]{hmt1}}\\
&\implies& c_ix\cdot c_iy\not=0 \hspace{0.5cm}\text{ by axiom (Ax3)}\\
&\implies& c_i(y\cdot c_ix)\not=0 \hspace{0.5cm}\text{ by axiom (Ax3)}\\
&\implies& y\cdot c_ix\not=0 \hspace{0.5cm}\text{ by \cite[Theorem 1.2.1]{hmt1}}\\
&\implies& y\leq c_ix \hspace{0.5cm}\text{ by the fact that $y\in At(\a{A})$}\\
&\implies& c_iy\leq c_ix \hspace{0.5cm}\text{ by \cite[Theorem 1.2.9]{hmt1}}.
\end{eqnarray*}
Note that all the cited theorems uses only the axioms (Ax0) - (Ax3).
\end{proof}
To represent an atomic algebra $\a{A}$, we roughly represent each atom $a\in \a{A}$ by a sequence $f$. Then, we show that $\a{A}$ can be embedded into the full algebra whose unit consists of all sequences representing atoms. The real challenge now is to arrange that the unit has the desired properties, by adding the substitutions $f\circ[i/j]$ or the transpositions $f\circ[i,j]$ whenever it is necessary. Such new sequences need to be associated to some atoms to keep the claim that each atom is represented by some sequences (maybe more than one), and there are no irrelevant sequences. For example, the following Lemma defines the substitutions of an atom, if any exists. 

\noindent {\bf{Notation:}} Let $i,j\in n$ be such that $i\not=j$. Define $t^i_ix\myeq x$ and $t^i_jx\myeq c_ix\cdot d_{ij}$
\begin{lem}\label{atom}
Suppose that $\a{A}\in\mathrm{RC}_n$ is an atomic algebra. Let $x\in At(\a{A})$ and let $i,j\in n$. Then, $t^i_jx\not=0\implies t^i_jx\in At(\a{A})$.
Moreover, if $x\leq d_{ij}$ then we have $t^i_jx=x$.
\end{lem}
\begin{proof}The statement is obvious if $i=j$, so we may assume that $i\not=j$. Suppose that $t^i_jx\not=0$. Since $\a{A}$ is atomic, one can find an atom $y\in At(\a{A})$ such that $y\leq t^i_jx=c_ix\cdot d_{ij}$. Thus, by Lemma~\ref{aib}, we have $c_ix=c_iy$. Note also that $y\leq d_{ij}$, so axiom (Ax6) implies 
$$t^i_jx=c_ix\cdot d_{ij}=c_iy\cdot d_{ij}=c_i(y\cdot d_{ij})\cdot d_{ij}=y.$$
Therefore, $t^i_jx$ is an atom in $\a{A}$ as desired. The remaining part is obvious. 
\end{proof}
\begin{defn}
Let $\a{A}\in\mathrm{DC}_n$ and let $\tau=[i_1/j_1]\circ\cdots\circ[i_m/j_m]\in\Omega_n$. For each $x\in \a{A}$, we define $\tau^{\a{A}}x\myeq t^{i_m}_{j_m}\cdots t^{i_1}_{j_1}x$. This is well defined by the following Lemma. 
\end{defn}
\begin{lem}\label{andreka}
Let $\a{A}\in\mathrm{DC}_n$ and let $\tau,\sigma\in\Omega_n$. Then 
$$(\forall i\in n) \ (\tau(i)=\sigma(i))\implies (\forall x\in\a{A}) \ (\tau^{\a{A}}x=\sigma^{\a{A}}x).$$
\end{lem}
\begin{proof}See \cite{andthm} proof of Lemma 1 therein. Axiom (Ax10) is used here. We note that And\'eka’s proof of this lemma is long, but using fairly obvious results on semigroups a much shorter proof can be provided.
\end{proof}
Similarly, we need to define the transpositions of an atom, if $\a{A}\in\mathrm{SC}_n$. This can be done using the axioms (Ax{11}) and (Ax{12}). By Lemma~\ref{trans} below, there is at least one choice (maybe many) to define these transpositions
\begin{defn}
Assume that $\a{A}\in\mathrm{SC}_n$. Let $i,j\in n$ be such that $i\not=j$ and let $x\in\a{A}$. We define $p_{ii}x\myeq x$ and $$p_{ij}x\myeq s^i_jc_jx\cdot s^j_ic_ix\cdot\prod_{k\in n,k\not=i,j}s^k_is^i_js^j_kc_kx.$$
If the product is empty (i.e. if $n=2$) then it is defined to be $1$.
\end{defn}
\begin{lem}\label{trans}
Let $\a{A}\in\mathrm{SC}_n$, let $i,j\in n$ and let $x\in At(\a{A})$. Then, $p_{ij}x\not=0$.
\end{lem}
\begin{proof}If $i=j$ then the statement is trivial. So we assume that $i\not=j$. If $n\geq 3$ then we are done by axioms (Ax{1}) and (Ax{12}). Suppose that $n=2$. If $x\leq -d_{01}$ then the desired follows by axiom (Ax{11}). If $x\leq d_{01}$ then 
\begin{eqnarray*}
p_{01}x&=& s^0_1c_1x\cdot s^1_0c_0x\\
&=& s^0_1c_1(x\cdot d_{01})\cdot s^1_0c_0(x\cdot d_{01})\\
&=& c_0(c_1(x\cdot d_{01})\cdot d_{01})\cdot c_1(c_0(x\cdot d_{01})\cdot d_{01})\\
&=& c_0(x\cdot d_{01})\cdot c_1(x\cdot d_{01}) \ \ \text{by axioms (Ax{2}) and (Ax{6})}\\
&=& c_0x\cdot c_1x.
\end{eqnarray*}
Hence, $x\leq c_0x\cdot c_1x=p_{01}x$ which implies that $p_{01}x\not=0$ as desired.
\end{proof}
We will use the following lemma in the proceeding sections. 
\begin{lem}\label{mylemma}
Suppose that $\a{A}\in\mathrm{SC}_n$ is complete and atomic algebra. Let $x,y,z\in At(\a{A})$ be some atoms and let $i,j,k\in n$ be such that $k\not=i,j$.
\begin{enumerate}
\item $y\leq s^k_is^i_js^j_kc_kx \text{ and } z=t^i_kt^j_it^k_jx\implies c_ky=c_kz$.
\item $y\leq s^i_jc_j x \text{ and } z=t^j_ix\implies c_iz=c_iy$.
\end{enumerate}
\end{lem}
\begin{proof} Suppose that $\a{A}\in\mathrm{SC}_n$ is complete and atomic algebra. Hence, by \cite[Theorem 1.2.6 and Theorem 1.5.3]{hmt1}, the operations $c_i$'s and $s^i_j$'s are completely additive. 
\begin{enumerate}
\item Suppose that $y\leq s^k_is^i_js^j_kc_kx$. Thus, by the assumptions on $\a{A}$, there are some atoms $a_0,a_1,a_2\in At(\a{A})$ such that 
$$a_0\leq c_kx, \ a_1\leq c_ja_0, \ a_2\leq c_ia_1, \ y\leq c_ka_2,$$
$$a_0\leq d_{jk}, \ a_1\leq d_{ij} \ a_2\leq d_{ki}.$$
Inductively, Lemma~\ref{atom} implies that 
$$a_0=t^k_jx, \ a_1=t^j_ia_0=t^j_it^k_jx, \ a_2=t^i_ka_1=t^i_kt^j_it^k_jx.$$ 
Therefore, $z=a_2$ and hence $y\leq c_kz$. The desired follows by Lemma~\ref{aib}.
\item Again, by assumptions, there is an atom $a$ such that $y\leq c_i a$, $a\leq d_{ij}$ and $a\leq c_jx$. Hence, by Lemma~\ref{atom}, we have $a=t^j_ix=z$. Therefore, $y\leq c_ia\leq c_i z$ and we are done by Lemma~\ref{aib}.\qedhere
\end{enumerate}
\end{proof}
\section{Networks and mosaics}
Throughout, let $\mathrm{K}\in\{ \mathrm{RC},\mathrm{DC},\mathrm{SC}\}$ and let $\a{A}\in\mathrm{K}_{n}$ be arbitrary but fixed. Suppose that $\a{A}$ is complete and atomic. The networks and mosaics we define here are approximations of the desired representation of $\a{A}$.
\begin{defn}
A pre-network is a pair $N=(N_1,N_2)$, where $N_1$ is a finite (possibly empty) set, and $N_2:{^nN_1}\rightarrow At(\a{A})$ is a partial map. We write $\n(N)$ for $N_1$ and $\e(N)$ for the domain of $N_2$. Also, we may write $N$ for any of $N$, $N_1$, $N_2$.
\end{defn}
\begin{enumerate}[1]
\item[-] We  write $\emptyset$ for the pre-network $(\emptyset,\emptyset)$.
\item[-] For the pre-networks $N$ and $N'$, we write $N\subseteq N'$ iff $\n(N)\subseteq \n(N')$, $\e(N)\subseteq \e(N')$, and $N'(f)=N(f)$ for all $f\in \e(N)$.
\item[-] Let $\alpha$ be an ordinal. A sequence of pre-networks $\langle N_{\beta}:\beta\in\alpha\rangle$ is said to be a chain if $N_{\gamma}\subseteq N_{\beta}$ whenever $\gamma\in\beta$. Supposing that $\langle N_{\beta}:\beta\in\alpha\rangle$ is a chain of pre-networks, define the pre-network $N=\bigcup\{N_{\beta}:\beta\in\alpha\}$ with $$\n(N)=\bigcup\{\n(N_{\beta}):\beta\in\alpha\}, \ \e(N)=\bigcup\{\e(N_{\beta}):\beta\in\alpha\}$$ and, for each $f\in \e(N)$, we let $N(f)=N_{\beta}(f)$, where $\beta\in\alpha$ is any ordinal with $f\in \e(N_{\beta})$.
\end{enumerate}
\begin{defn}
Let $N$ be a pre-network and let $f,g\in\e(N)$. A sequence of edges $h_0,\ldots,h_m$ is said to be a zigzag of length $m$ from $f$ to $g$ in $N$ if the following hold:
\begin{itemize}
\item $h_0=f$ and $h_m=g$.
\item For each $0\leq t\leq m$, $\r(f)\cap\r(g)\subseteq \r(h_t)$.
\item For each $0\leq t<m$, there is $i_{t+1}\in n$ such that $$h_t\not=h_{t+1}, \ \ h_t\equiv_{i_{t+1}}h_{t+1} \ \text{ and } \ c_{i_{t+1}}N(h_t)=c_{i_{t+1}}N(h_{t+1}).$$
\end{itemize}
\end{defn}
\begin{defn}\label{netwrok}
A pre-network $N$ is said to be a network if it satisfies the following conditions for each $f,g\in \e(N)$ and each $i,j\in n$:
\begin{enumerate}
\item[(a)] \begin{enumerate}[(i)]
\item $\mathrm{K}=\mathrm{DC}\implies\e(N)\text{ is diagonalizable}$.
\item $\mathrm{K}=\mathrm{SC}\implies\e(N)\text{ is diagonalizable and permutable}$.
\end{enumerate}
\item[(b)] $N(f)\leq d_{ij}\iff f(i)=f(j)$.
\item[(c)] If $0<\vert\r(f)\cap \r(g)\vert<n$ then there is a zigzag from $f$ to $g$ in $N$. 
\end{enumerate}
\end{defn}
\begin{lem}\label{tauclosed}
Let $N$ be a network. Let $f,g\in\e(N)$, $i,j\in n$ and $\tau\in\Omega_n$. The following are true:
\begin{enumerate}
\item If $f\equiv_i g$ then $c_iN(f)=c_iN(g)$.\footnote{The proof of this item is distilled from Andr\'eka's proof in \cite[Theorem 9.4]{monk}.} 
\item If $f\circ[i/j]\in\e(N)$ then $N(f\circ[i/j])=t^i_jN(f)$. 
\item Suppose that $\mathrm{K}\in\{\mathrm{DC},\mathrm{SC}\}$. Then, $N(f\circ\tau)=\tau^{\a{A}}N(f)$.
\end{enumerate} 
\end{lem}
\begin{proof}Let $N,f,g,i,j,\tau$ be as required.
\begin{enumerate}
\item Suppose that $f\equiv_ig$. If $f=g$ then we are done. Assume that $f\not=g$. So, it follows that $0<\vert \r(f)\cap\r(g)\vert <n$. Thus, we can assume that $h_0,\ldots,h_m\in\e(N)$ is a zigzag from $f$ to $g$.

We will define $i_1,\ldots,i_m,j_1,\ldots,j_m,k_1,\ldots,k_m$ such that $$N(g)\leq s^{i_m}_{j_m}c_{k_m}\cdots s^{i_1}_{j_1}c_{k_1}N(f)\cdot\prod_{l\in J}d_{l\tau(l)},$$
where $i_1,\ldots,i,\tau,J$ satisfy the conditions of axiom (Ax7). Hence, by axiom (Ax7) again, $N(g)\leq c_iN(f)$ and we will be done. We define the $i$'s as follows. For each $1\leq t\leq m$, let $i_t\in n$ be such that $h_{t-1}\equiv_{i_t}h_{t}$. Let $J=\{i_1,\ldots,i_m\}\setminus\{i\}$ and $J^{+}=J\cup\{i\}$. Note that $\vert J^{+}\vert>\vert J\vert$. We will define $j_t$ and $k_t$ for $1\leq t\leq m$ by induction on $t$ such that by letting $$\tau_t=[i_t/j_t]\circ\cdots\circ[i_1/j_1]$$ 
we will have for all $t<m$ that 
\begin{eqnarray*}
h_0(l)&=&h_{t+1}(\tau_{t+1}(l))\text{ for all }l\in J,\\
N(h_{t+1})&\leq&s^{i_{t+1}}_{j_{t+1}}c_{k_{t+1}}N(h_t)\text{ and }\\
k_{t+1}&\in&J^{+}\setminus\tau_t^{*}J.
\end{eqnarray*}
Let $t<m$ and assume that $j_{t'},k_{t'}$ have been define for all $1\leq t'\leq t$ with the above properties. Now, we have two cases.

\noindent{\bf{Case 1:}}
Suppose that $h_t(i_{t+1})\in Rng(h_{t+1})$, say $h_t(i_{t+1})=h_{t+1}(j)$, for some $j\in n$. Note that $h_{t+1}\not=h_t$ and $h_t(j)=h_{t+1}(j)=h_t(i_{t+1})$, then $j\not= i_{t+1}$. Hence, $N(h_{t+1})\leq s^{i_{t+1}}_jN(h_t)$. We let $j_{t+1}\myeq j$ and $k_{t+1}\in (J^{+}\setminus \tau_t^{*}J)$ be arbitrary, note that $(J^{+}\setminus \tau_t^{*}J)\not=\emptyset$.% by the fact that $\vert J^{+}\vert >\vert J\vert$.

\noindent{\bf{Case 2:}}
Suppose that $h_t(i_{t+1})\not\in Rng(h_{t+1})$. Let $j_{t+1}\myeq k_{t+1}\myeq i_{t+1}$. Let $l\in J$ be arbitrary. Then, by $l\not=i$, we have $$h_t(\tau_t(l))=h_0(l)\in \r(f)\cap\r(g)\subseteq \r(h_{t+1}).$$ Hence, $i_{t+1}\not=\tau_t(l)$ as desired.

It is not difficult to check that the above choices satisfy our requirements. Then $N(g)\leq s^{i_m}_{j_m}c_{k_m}\cdots s^{i_1}_{j_1}c_{k_1}N(f)$. Also, for any $l\in J$, $$g(l)=f(l)=h_0(l)=h_m(\tau_m(l))=g(\tau(l)).$$
Hence $N(g)\leq d_{l\tau(l)}$ for all $l\in J$ by condition (b) of networks. Therefore, $N(g)\leq s^{i_m}_{j_m}c_{k_m}\cdots s^{i_1}_{j_1}c_{k_1}N(f)\cdot \prod_{l\in J} d_{l\tau(l)}\leq c_iN(f)$.
\item If $i=j$ then we are done. So, we may suppose that $i\not=j$ and we also assume that $f\circ[i/j]\in\e(N)$. Thus, by the above item, $c_iN(f)=c_iN(f\circ[i/j])$. Also $N$ is a network, then $N(f\circ[i/j])\leq d_{ij}$. Therefore, $N(f\circ[i/j])\leq c_iN(f)\cdot d_{ij}=t^i_jN(f)$.
\item Let $\tau\in\Omega_n$ and assume that $\tau=[i_1/j_1]\circ\cdots\circ[i_m/j_m]$. The statement follows by an induction argument that uses item (2). \qedhere
\end{enumerate}
\end{proof}
An arbitrary sequence $f$ (of length $n$) is said to be a repetition free sequence if and only if $\vert\r(f)\vert=n$.
\begin{defn}
Let $a\in At(\a{A})$ and let $f$ be any sequence of length $n$ such that $(\forall i,j\in n) \ (f(i)=f(j)\iff a\leq d_{ij})$. The mosaic generated by $f$ and $a$, in symbols $M(f,a)$, is defined to be the pre-network $N$, where $\n(N)=\r(f)$ and:
\begin{enumerate}[(a)]
\item If $\mathrm{K}=\mathrm{RC}$ then $\e(N)=\{f\}$ and $N(f)=a$. 
\item\label{it} If $\mathrm{K}=\mathrm{DC}$ then $\e(N)=\{f\circ\tau:\tau\in\Omega_n\}$, $N(f)=a$ and, for each $\tau\in\Omega_n$, $N(f\circ\tau)=\tau^{\a{A}}a$.
\item If $\mathrm{K}=\mathrm{SC}$ and $\vert\r(f)\vert<n$ then the mosaic is defined in the same way as in item \eqref{it} above.
\item Suppose that $\mathrm{K}=\mathrm{SC}$ and assume that $f$ is a repetition free sequence. Then $\e(N)=\{f\circ\tau:\tau\in\Delta_n\}$, and the labeling is defined as follows. First, $N(f)=a$ and for each $\tau\in\Omega_n$, $N(f\circ\tau)=\tau^{\a{A}}a$. So, each non-repetition free sequence is labeled. It remains to label the repetition free sequences. Let $g_0,g_1, \ldots, g_N$ ($N\geq 1$ because $n\geq 2$) be some enumeration of the repetition free sequences such that 
\begin{eqnarray*}
g_0&=&f\\
g_i&=&g_j\circ[k,l]\text{ for some }j<i \text{ and }k\not=l\\
g_i&\not=&g_j \text{ if }i\not=j
\end{eqnarray*}
Such enumeration is possible. Assume that $N(g_j)$ is defined for all $j<i$. Let $j<i$ and $k\not=l$ be such that $g_i=g_j\circ [k,l]$ (If there are several such $j,k,l$ then we just select one such triple). Now we choose any atom $b\leq p_{kl}N(g_j)$ (such an atom exist by Lemma~\ref{trans}) and we define $N(g_i)=b$.
\end{enumerate}
We check that the mosaic $M(f,a)$ is well defined. The essential part follows from Lemma~\ref{andreka}. Suppose that $\tau,\sigma\in\Omega_n$ are such that $f\circ\tau=f\circ\sigma$. Fix an enumeration $\tau_1,\ldots,\tau_m$ of the set $\{[i/j]:i,j\in n\text{ and }f(i)=f(j)\}$. Now, let $\gamma=\tau_0\circ\cdots\circ\tau_m$. Hence, $\gamma\circ\tau=\gamma\circ\sigma$ and by Lemma~\ref{atom} we must have $\tau^{\a{A}}(\gamma^{\a{A}}a)=\sigma^{\a{A}}(\gamma^{\a{A}}a)$. But by the condition on the sequence $f$ and by Lemma~\ref{atom}, $\gamma^{\a{A}}a=a$. Therefore, $\tau^{\a{A}}a=\sigma^{\a{A}}a$. 
\end{defn}
\begin{lem}\label{mosok}
Let $a\in At(\a{A})$ and let $f$ be any sequence of length $n$ such that $$(\forall i,j\in n) \ (f(i)=f(j)\iff a\leq d_{ij}).$$
Then the mosaic $M\myeq M(f,a)$ is actually a network.
\end{lem}
\begin{proof}
Suppose that $\mathrm{K}=\mathrm{G}_n$ and suppose that $f$ is a repetition free sequence. We prove the above statement for this case only, the other cases are similar so we omit their details. The construction of $M$ guarantees that it satisfies conditions (a) and (b) of networks. It remains to show that $M$ also satisfies condition (c). First, we prove the following. 
\begin{equation}\label{cond}
(\forall g\in \e(M)) \ (\forall i,j\in n) \ \ M(g(i/g(j)))=t^i_jM(g).
\end{equation}
Let $i,j\in n$ and let $g\in\e(M)$ be such that $i\not=j$. If $f=g$ then by the definition of $M$ we have $$M(g(i/g(j)))=M(f(i/f(j)))=M(f\circ[i/j])=t^i_jM(f)=t^i_jM(g).$$ Suppose that there is a transformation $\tau\in\Omega_n$ such that $g=f\circ\tau$. Then, $g(i/g(j))=g\circ[i/j]=f\circ\tau\circ[i/j]$. So, by the construction of $M$, we have $M(g(i/g(j)))= (\tau\circ[i/j])^{\a{A}}a=t^i_j \tau^{\a{A}}a=t^i_j M(g)$. Recall the enumeration $g_0,g_1,\ldots, g_N$ of the repetition free sequences given in the definition of $M$. To finish proving \eqref{cond}, it remains to show that
\begin{equation}\label{ind}
(\forall m\in N+1) \ \ M(g_m(i/g_m(j)))=t^i_jM(g_m).
\end{equation}
We do this by induction on $m$. If $m=0$ then we are done as $g_0=f$. Let $1\leq p\leq N$. Assume that \eqref{ind} is true for all $m<p$. We will show that \eqref{ind} is true for $p$, too. Let $g_p=g_m\circ[k,l]$ be such that $m<p$ and suppose that we choose $M(g_p)\leq p_{kl}M(g_m)$.

\noindent{{\bf{Case 1:}} $i\not=k,l$ and $j=k$ or $j=l$.} 

Assume first that $j=l$. Consider the edges:
$$h_1=g_m\circ[i/k], \ \ h_2=g_m\circ[i/k]\circ[k/l] \ \text{ and } \ h_3=g_m\circ[i/k]\circ[k/l]\circ[l/i].$$
Thus, $h_3\equiv_ig_p$ because $g_p=g_m\circ [k,l]$. Hence, $g_p(l)=h_3(l)=h_3(i)$, which means that $h_3=g_p(i/g_p(j))$. By induction hypothesis, $M(h_1)=t^i_kM(g_m)$. None of the sequences $h_1, h_2$ and $h_3$ is repetition free, so we already showed that $$M(g_p(i/g_p(j)))=M(h_3)=t^l_iM(h_2)=t^l_it^k_lM(h_1)=t^l_it^k_lt^i_kM(g_m).$$
By $M(g_p)\leq p_{kl}M(g_m)$ we have $M(g_p)\leq s^i_ks^k_ls^l_ic_i M(g_m)$. Since in $\mathrm{K}_n$ the so-called the Merry Go Round equation $s^i_ks^k_ls^l_ic_i x=s^i_ls^l_ks^k_ic_i x$ is true, then we also have that $M(g_p)\leq s^i_ls^l_ks^k_ic_i M(g_m)$. Now, Lemma~\ref{mylemma} implies that $M(g_p(i/g_p(j)))\leq c_iM(g_p)$. By condition (a) of the networks, we have $M(g_p(i/g_p(j)))\leq c_iM(g_p)\cdot d_{ij}=t^i_jM(g_p)$. The case $j=k$ is completely similar, except that we do not have to use the Merry Go Round equation. 

\noindent{{\bf{Case 2:}} $i\not=k,l$ and $j\not=k,l$.} 

Then $[i/l]\circ[i/j]=[i/j]$. Hence, $g_p(i/g_p(j))=g_p\circ[i/j]=g_p\circ[i/l]\circ[i/j]$. We can use the previous case to conclude that $M(g_p\circ[i/l])=t^i_lM(g_p)$. Note that $g_p\circ[i/l]$ is not repetition free, so we know that $$M(g_p(i/g_p(j)))=M(g_p\circ[i/l]\circ[i/j])=t^i_jt^i_lM(g_p)=t^i_jM(g_p),$$ the last equality follows from Lemma~\ref{andreka}.

\noindent{{\bf{Case 3:}} $i=k$ and $j=l$.} 

Then $g_p(i/g_p(j))=g_p\circ[i/j]=g_p\circ[k/l]=g_m\circ[l/k]$. Thus, by the induction hypothesis, we have $M(g_p(i/g_p(j)))=t^l_kM(g_m)$. We also have $M(g_p)\leq s^i_jc_j M(g_m)$ because of the fact that $M(g_p)\leq p_{kl}M(g_m)$. Hence, by Lemma~\ref{mylemma}, $M(g_p(i/g_p(j)))\leq c_iM(g_p)$. Again, since $M(g_p(i/g_p(j)))\leq d_{ij}$, $M(g_p(i/g_p(j)))=c_iM(g_p)\cdot d_{ij}=t^i_jM(g_p)$. The case $i=k$ and $j\not=l$ is as above. The case $i=l$ is completely analogous.

Therefore, \eqref{ind} is true and thus we have shown that \eqref{cond} is also true. We need also to show that the following holds:
\begin{equation}\label{2}
(\forall \tau\in\Omega_n) \ \text{ there is a zigzag from }f\circ\tau\text{ to }f\text{ in }M. 
\end{equation}
Let $\tau\in\Omega_n$. So, we can suppose that $\tau=[i_1/j_1]\circ\cdots\circ[i_m/j_m]$ with the smallest possible $m$. Consider the edges:
$$h_0=f, \ h_1=h_0\circ[i_1/j_1], \ h_2=h_1\circ[i_2/j_2], \ \ldots, \ h_m=h_{m-1}\circ[i_m/j_m]=f\circ\tau.$$
By \eqref{cond}, it is easy to see that $h_0,\ldots,h_m$ is a zigzag from $f\circ\tau$ to $f$ in $M$. Hence, \eqref{2} is proved. 
Now we are ready to show that the mosaic $M$ satisfies condition (c) in Definition~\ref{netwrok}. Let $g,h\in\e(M)$ be such that $0<\vert\r(g)\cap\r(h)\vert<n$. 
\begin{enumerate}[(I)]
\item Suppose that none of $g$ and $h$ is repetition free. By \eqref{2} we have the following:
\begin{enumerate}
\item[-] There is a zigzag $h_0,\ldots,h_m$ from $g$ to $f$ in $M$.
\item[-] There is a zigzag $w_0,\ldots,w_d$ from $h$ to $f$ in $M$.
\end{enumerate}
It is not hard to see that $h_0,\ldots,h_m=f=w_d,\dots,w_0$ is a zigzag from $g$ to $h$ in $M$. This true because of the facts $\r(g)\subseteq\r(f)$ and $\r(h)\subseteq\r(f)$.
\item Suppose that one of $g$ and $h$ is repetition free, say $g$. Thus, $h$ must not be repetition free, by the assumption $0<\vert\r(g)\cap\r(h)\vert<n$. Also, we can find an element $u\in \r(g)\setminus\r(h)$. Let $i\in n$ be such that $g(i)=u$ and let $j\in n\setminus\{i\}$ (this $j$ exists because $n\geq 2$). Clearly, $\r(g\circ[i/j])\cap\r(h)=\r(g)\cap\r(h)$. Hence, by (I) above, there is a zigzag $h_0,\ldots,h_m$ from $g\circ[i/j]$ to $h$ in $M$. But \eqref{cond} and Lemma~\ref{aib} imply $c_iM(g)=c_iM(g\circ[i/j])$. Thus, $g,h_0,\ldots,h_m$ is a zigzag from $g$ to $h$ in $M$.
\end{enumerate}
Therefore, the mosaic $M$ is indeed a network.
\end{proof}
\section{Games and representability}
The games we use here are games of infinite lengths between two players $\forall$ and $\exists$. These games are basically Banach-Mazur games (see \cite[Problem 43]{game2} and \cite{game1}) in disguise.
\begin{defn}
Let $\alpha$ be an ordinal. We define a game, denoted by $G_{\alpha}(\a{A})$, with $\alpha$ rounds, in which the players $\forall$ and $\exists$ build a chain of pre-networks $\langle N_{\beta}:\beta\in\alpha\rangle$ as follows. In round $0$, $\exists$ starts by letting $N_0=\emptyset$. Suppose that we are in round $\beta\in\alpha$ and assume that each $N_{\lambda}$, $\lambda\in\beta$, is a pre-network. If $\beta$ is a limit ordinal then $\exists$ defines $N_{\beta}=\bigcup\{N_{\lambda}:\lambda\in\beta\}$. If $\beta=\gamma+1$ is a successor ordinal then the players move as follows:
\begin{enumerate}[(a)]
\item $\forall$ chooses an atom $b\in At(\a{A})$, $\exists$ must respond with a pre-network $N_{\beta}\supseteq N_{\gamma}$ containing an edge $g$ with $N_{\beta}(g)=b$.
\item Alternatively, $\forall$ chooses an edge $g\in \e(N_{\gamma})$, an index $i\in n$ and an atom $b\in At(\a{A})$ such that $N_{\gamma}(g)\leq c_ib$. In this case, $\exists$ must respond with a pre-network $N_{\beta}\supseteq N_{\gamma}$ such that for some $u\in\n(N_{\beta})$ we have $g(i/u)\in\e(N_{\beta})$ and $N_{\beta}(g(i/u))=b$.
\end{enumerate}
$\exists$ wins if each pre-network $N_{\beta}$, $\beta\in\alpha$, played during the game is actually a network. Otherwise, $\forall$ wins. There are no draws. 
\end{defn}
What we have defined are the rules of the game. There are many different matches of the game that satisfy the rules. The idea of the game is that the current network is refined in the current round of the game. $\forall$ can challenge $\exists$ to find a suitable refinement in any of the two ways he likes, and she must either do or lose. Since the game then continues from there, her response must itself be refinable in any way, if she is not to lose.

\begin{prop}\label{winning}
Let $\alpha$ be an ordinal. $\exists$ has a winning strategy in the game $G_{\alpha}(\a{A})$. 
\end{prop}
\begin{proof}
Let $\alpha$ be an ordinal and let $\beta\in \alpha$. Clearly, $\exists$ always wins in the round $\beta$ if $\beta=0$ or $\beta$ is a limit ordinal. So, we may suppose that $\beta=\gamma+1$ is a successor ordinal. We also may assume inductively that $\exists$ has managed to guarantee that $N_{\gamma}$ is a network. We consider the possible moves that $\forall$ can make.
\begin{enumerate}[(a)]
\item Suppose that $\forall$ chooses an atom $b\in At(\a{A})$. If there is an edge in $N_{\gamma}$ with the required conditions, then $\exists$ lets $N_{\beta}=N_{\gamma}$. Otherwise, she picks brand new nodes $f_0,\ldots,f_{n-1}$ such that $$(\forall i,j\in n) \ \ (f_i=f_j\iff b\leq d_{ij}).$$
Let $f=(f_0,\ldots, f_{n-1})$. She defines $N_{\beta}=N_{\gamma}\cup M(f,a)$, where $M(f,a)$ is the mosaic generated by $f$ and $a$. Note that our construction guarantees that $\n(N_{\gamma})\cap\n(M(f,a))=\emptyset$, so $N_{\beta}$ is well defined. Now, Lemma~\ref{mosok} guarantees that $N_{\beta}$ is a network.
\item Alternatively, suppose that $\forall$ chooses an edge $f\in\e(N_{\gamma})$, an index $i\in n$ and an atom $b\in At(\a{A})$ such that $N_{\gamma}(f)\leq c_ib$. Again, if there is a node $u\in\n(N_{\gamma})$ such that $f(i/u)\in\e(N_{\gamma})$ and $N_{\gamma}(f(i/u))=b$, then $\exists$ lets $N_{\beta}=N_{\gamma}$. Suppose that such $u$ does not exist. $\exists$'s strategy goes as follows. Let $a=N_{\gamma}(f)$, so we have $c_ia=c_ib$ by Lemma~\ref{aib}. 

Suppose that there is some $j\in n\setminus\{i\}$ such that $b\leq d_{ij}$. Then $b\leq c_ia\cdot d_{ij}$ and thus $b=t^i_ja$. If $f(i/f_j)\in\e(N_{\gamma})$ then we have $N_{\gamma}(f(i/f_j))=t^i_ja=b$, which contradicts the assumptions. So, we may assume that $f(i/f_j)\not\in\e(N_{\gamma})$ (hence, $\mathrm{K}_n=\mathrm{Crs}_n$ because $N_{\gamma}$ is a network). $\exists$ defines $N_{\beta}=N_{\gamma}\cup M(f(i/f_j),b)$. Thus, $N_{\beta}$ is indeed well defined and it is not hard to see that it is a network.

Suppose that $b\not\leq d_{ij}$ for every $j\in n\setminus\{i\}$. In this case, $\exists$ picks brand new node $u$ and then she lets $N_{\beta}=N_{\gamma}\cup M(f(i/u),b)$. We need to check that $N_{\beta}$ is well defined. Let $g\in \e(N_{\gamma})\cap\e(M(f(i/u),b))$. Remember that $u$ was a brand new node, so the existence of such $g$ means that $\mathrm{K}\in\{\mathrm{D},\mathrm{G}\}$ and that there are $j\in n\setminus\{i\}$ and $\tau\in\Omega_n$ such that $g=f(i/u)\circ[i/j]\circ\tau$. Hence, $g=f\circ[i/j]\circ\tau$. By Lemma~\ref{tauclosed}, it follows that $N_{\gamma}(g)=\tau^{\a{A}}t^i_ja=\tau^{\a{A}}(c_ia\cdot d_{ij})$. Let $M\myeq M(f(i/u),b)$, by the  construction of the mosaic $M$, we also have $M(g)=\tau^{\a{A}}t^i_jb=\tau^{\a{A}}(c_ib\cdot d_{ij})$. It remains to see that $N_{\gamma}(g)=M(g)$, which is true  because $c_ia=c_ib$. Thus, $\exists$ could manage to guarantee that $N_{\beta}$ is well defined. Now, we show that $N_{\beta}$ is actually a network. 

Lemma~\ref{mosok} and the induction hypothesis guarantee that $N_{\beta}$ satisfies conditions (a) and (b) of networks. For the same reasons, it is apparent that $N_{\beta}$ satisfies condition (c) for any two edges that both lie in $N_{\gamma}$ or in the mosaic $M$. Let $g\in\e(N_{\gamma})$ and let $h\in\e(M)\setminus\e(N_{\gamma})$. Suppose that $0<\vert\r(g)\cap\r(h)\vert<n$. We need to find a zigzag from $g$ to $h$ in the pre-network $N_{\beta}$.

Note that $u\in \r(h)$. Let $j\in n$ be such that $h(j)=u$ and choose an index $k\in n\setminus\{j\}$. Let $h'=h\circ[j/k]$. Thus, $h'\in\e(N_t)$ and $\r(h)\cap\r(g)=\r(h')\cap\r(g)$. By the induction hypothesis, there is a zigzag $w_0,\ldots,w_m$ from $g$ to $h'$ in $N_{\beta}$. Remember the fact that $N_{\beta}(h')=M(h')=t^j_kM(h)=t^j_kN_{\beta}(h)$. Hence, by Lemma~\ref{aib}, we have $c_jN_{\beta}(h')=c_jN_{\beta}(h)$. Therefore, $w_0,\ldots,w_m,h$ is a zigzag from $g$ to $h$ in $N_{\beta}$ as required. 
\end{enumerate}
Therefore, if $\exists$ plays according to the strategy above she can win any play of the game $G_{\alpha}(\a{A})$, regardless of what moves $\forall$ makes.
\end{proof}
%Roughly, the strategy is to add (if necessary) a new mosaic in each step. The axioms are then used to label each element of the mosaic by some relevant atom. 
%It is true that Andr\'eka's methods in \cite{monk} and \cite{andthm} could also suggest a strategy for $\exists$ to win the game $G_{\alpha}(\a{A})$. Despite the technical similarities, we note that this suggested strategy is quite different than our strategy. Andr\'eka's method would suggest adding a sequence in each step, not a mosaic, then at the end of the construction the axioms can be used to make sure that the mosaics generated by these sequences are eventually added.  
Roughly, $\exists$'s winning strategy goes as follows. In each step, she adds a whole mosaic, and she uses the axioms to label the elements of the mosaics by relevant atoms. Now we will use this to show that the algebra $\a{A}$ can be represented into a full algebra whose unit is a union of mosaics. 
\begin{prop}\label{prop}
$\a{A}$ is representable, i.e. there is $\a{B}\in\mathrm{Ks}_n$ such that $\a{A}\cong\a{B}$.
\end{prop}
\begin{proof}
Let $\alpha$ be an ordinal (large enough) and consider a play $\langle N_{\beta}:\beta\in\alpha\rangle$ of $G_{\alpha}(\a{A})$ in which $\exists$ plays as in Proposition~\ref{winning}, and $\forall$ plays every possible move at some stage of play. That means,
\begin{enumerate}[(G 1)]
\item each atom $a\in At(\a{A})$ is played by $\forall$ in some round, and
\item for every $b\in At(\a{A})$, every $\beta\in\alpha$, and each $f\in \e(N_{\beta})$ and every $i\in n$ with $N_{\beta}(f)\leq c_ib$, $\forall$ plays $b$, $i$, $f$ in some round.
\end{enumerate}
Let $U=\bigcup\{\n(N_{\beta}):\beta\in\alpha\}$ and let $V=\bigcup\{\e(N_{\beta}):\beta\in\alpha\}\subseteq {^nU}$. By condition (a) of Definition~\ref{netwrok}, we can guarantee that $\a{P}(V)\in\mathrm{Ks}_n$. Thus, it remains to show that $\a{A}$ is embeddable into $\a{P}(V)$. For this, we define the following function. For each $x\in \a{A}$, let
$$\Psi(x)=\{f\in V: \exists\beta\in\alpha \ (f\in \e(N_{\beta}) \text{ and }N_{\beta}(f)\leq x)\}.$$
It is not hard to see that $\Psi$ is a Boolean homomorphism. Also, (G 1) above implies that $\Psi$ is one-to-one. We check cylindrifications and diagonals. Let $x\in\a{A}$. Then, for each $f\in V$, we have
\begin{eqnarray*}
f\in \Psi(c_ix)&\iff& \exists\beta\in\alpha\left( f\in N_{\beta} \text{ and }N_{\beta}(f)\leq c_ix\right)\\
&\iff& \exists g\in V \ \exists\beta\in\alpha \ (f,g\in N_{\beta}, f\equiv_i g, N_{\beta}(g)\leq x)\\
&\iff& \exists g\in V \left( f\equiv_i g\text{ and }g\in \Psi(x)\right)\\
&\iff& f\in C_i^{[V]}\Psi(x).
\end{eqnarray*}
The second $\iff$ follows by (G 2) and Lemma~\ref{tauclosed} (1). For the diagonals, let $i,j\in n$, then by the second condition of networks
\begin{eqnarray*}
\Psi(d_{ij})&=&\{f\in V: \exists\beta\in\alpha \ (f\in \e(N_{\beta}) \text{ and }N_{\beta}(f)\leq d_{ij})\}\\
&=&\{f\in V:f(i)=f(j)\}\\
&=&D_{ij}^{[V]}.
\end{eqnarray*}
Therefore, $\a{A}$ is isomorphic to a subalgebra $\a{B}\subseteq\a{P}(V)$ as desired. 
\end{proof}
%Now, we are ready to proof our main theorem.
\begin{proof}[Proof of Main Theorem~\ref{main}] Let $\mathrm{K}\in\{\mathrm{RC},\mathrm{DC},\mathrm{SC}\}$. It is easy to see that every concrete algebra in $\mathrm{Ks}_n$ satisfies the axioms defining $\mathrm{K}_n$. Conversely, let $\a{A}\in\mathrm{K}_n$, we need to show that $\a{A}\in\mathbf{I}\mathrm{Ks}_n$. By \cite[Theorem 2.18]{tarski}, $\a{A}$ can be embedded into a complete and atomic $\a{A}^{+}\in\mathrm{K}_n$. Hence, Proposition~\ref{prop} implies that $\a{A}^{+}$ is representable. Therefore, $\a{A}$ is representable as well. 
\end{proof}
\begin{mthm}Let $n\geq 2$ be a finite ordinal and let $\mathrm{K}\in\{\mathrm{RC},\mathrm{DC},\mathrm{SC}\}$. Then, every $\a{A}\in\mathrm{K}_n$ is completely representable into an algebra $\a{B}\in\mathbf{I}\mathrm{Ks}_n$. That means, $\a{A}$ is embeddable into $\a{B}$ via an embedding that preserves infinite sums and infinite products. 
\end{mthm}
\begin{proof}
See the proofs of Proposition~\ref{prop} and Main Theorem~\ref{main}. 
\end{proof}
\section{Other networks}
If $n=2$ then one can easily see that (Ax7) follows from (Ax0)-(Ax6). Suppose that $n\geq 3$. H. Andr\'eka and I. N\'emeti showed that the class $\mathrm{RCs}_n$ can not be characterized by finitely many equations, c.f. \cite[Theorem 5.5.13]{hmt2} and \cite[Theorem 9.3]{monk}, thus in this case (Ax7) can not be replaced with a finite set of axioms. However, in \cite{andthm} and \cite{and2000}, axiom (Ax7) was omitted from the characterizations of the classes $\mathrm{DCs}_n$ and $\mathrm{SCs}_n$, and hence it was shown that these classes are finitely axiomatizable.

Although, axiom (Ax7) is essential in the proof presented in the previous sections, this axiom can be omitted from the definitions of $\mathrm{DC}_n$ and $\mathrm{SC}_n$. It could be possible to give a syntactical proof for this fact, but we prefer to introduced some modifications of the networks that would allow us to proceed without the need of (Ax7). Let $\mathrm{DC}_n^{-}$ and $\mathrm{SC}_n^{-}$ be the resulting classes after deleting axiom (Ax7) from the definitions of $\mathrm{DC}_n$ and $\mathrm{SC}_n$, respectively. Suppose that $\mathrm{K}\in\{\mathrm{DC}^{-},\mathrm{SC}^{-}\}$ and let $\a{A}\in\mathrm{K}_n$ be complete and atomic. We show that $\a{A}$ is representable.
\begin{defn}
A `modified' network $N$ is a pre-network that satisfies the following conditions for each $f\in\e(N)$ and each $i,j\in n$.
\begin{enumerate}[(a)]
\item $\e(N)$ is diagonalizable, and $\e(N)$ is permutable iff $\mathrm{K}=\mathrm{SC}^{-}$.   
\item $N(f)\leq d_{ij}\iff f(i)=f(j)$.
\item $N(f\circ[i/j])=t^i_jN(f)$.
\end{enumerate}
\end{defn}
Let $N$ be a modified network. Let $i\in n$ and let $f,g\in \e(N)$ be such that $f\equiv_ig$. We prove that $c_iN(f)=c_iN(g)$. Let $j\in n$ be such that $j\not=i$. Note that $g\circ[i/j]=f\circ[i/j]\in\e(N)$. By condition (c) of the above definition and Lemma~\ref{aib}, it follows that 
$$c_iN(f)=c_iN(f\circ[i/j])=c_iN(g\circ[i/j])=c_iN(g).$$
Hence, we have shown that $N$ satisfies all the items of Lemma~\ref{tauclosed}. Note that we did not use axiom (Ax7). Let $\alpha$ be an ordinal. We also define the modified game $G'_{\alpha}(\a{A})$ to be as same as $G_{\alpha}(\a{A})$ except that the players now have to build modified networks. It is not hard to see that if $\exists$ plays with the same strategy given in Proposition~\ref{winning} then she will win any game regardless of what moves $\forall$ can make. It is quite straightforward to see that the mosaics in this case are also modified networks (see the proof of \eqref{cond}). Now the proof of Proposition~\ref{prop} works verbatim to show that $\a{A}$ is representable. 

We note that the modified networks and the modified games can be used also to prove the representability of some variants of Pinter's algebras and the quasi polyadic algebras. For more details, one can see \cite{rel}. Games and networks were also used to reprove similar representability results for variants of relation algebras \cite[Theorem 7.5]{HHB} and \cite{me}.

\end{document}